\RequirePackage{fix-cm}
\documentclass[smallcondensed,numbook,envcountsame]{svjour3}     

\usepackage{graphicx}

\usepackage{amsmath,epsfig}
\usepackage{cases} 
\usepackage{amssymb,cite}
\usepackage{epsfig}
\usepackage{color}
\usepackage{bm}
\usepackage{cite}
\usepackage{multirow}
\usepackage[utf8]{inputenc}
\usepackage[algo2e,linesnumbered, vlined,ruled]{algorithm2e}

\newcommand{\X}{{\cal X}}

\newcommand{\be}{\begin{equation}}
\newcommand{\ee}{\end{equation}}

\newcommand{\LCal}{\mathcal{L}}

\begin{document}

\title{
On the non-ergodic convergence rate of an inexact augmented Lagrangian framework for composite convex programming
}

\titlerunning{Non-ergodic convergence rate of an IAL framework for convex programming}        

\author{Ya-Feng Liu \and Xin Liu \and Shiqian Ma}


\institute{Y.-F. Liu \at
              State Key Laboratory
of Scientific and Engineering Computing,\\
Institute of Computational
Mathematics and Scientific/Engineering Computing,\\
 Academy of Mathematics and Systems Science, \\
Chinese Academy of Sciences, China \\
              \email{yafliu@lsec.cc.ac.cn}
           \and
           X. Liu \at
              State Key Laboratory of Scientific and Engineering
		Computing,\\
 Academy of Mathematics and Systems Science,\\
  Chinese Academy of Sciences and University of Chinese Academy of Sciences, China\\
              \email{liuxin@lsec.cc.ac.cn}
           \and
           S. Ma \at
              {Department of Mathematics, University of California, Davis, CA 95616, USA \\
              \email{sqma@math.ucdavis.edu}}
}

\date{Received: date / Accepted: date}

\maketitle

\begin{abstract}
In this paper, we consider the linearly constrained composite convex optimization problem, whose objective is a sum of a smooth function and a possibly nonsmooth function. We propose an inexact augmented Lagrangian (IAL) framework for solving the problem. {The stopping criterion used in solving the augmented Lagrangian subproblem in the proposed IAL framework is weaker and potentially much easier to check than the one used in most of the existing IAL frameworks/methods.} {We analyze the global convergence and the non-ergodic convergence rate of the proposed IAL framework. {Preliminary numerical results are presented to show the efficiency of the proposed IAL framework and the importance of the non-ergodic convergence and convergence rate analysis.}}




\end{abstract}


\section{Introduction}
{In this paper, we consider the linearly constrained composite convex optimization problem
\begin{equation}\label{problemcouple}
\begin{array}{cl}
\displaystyle \min_{x\in\mathbb{R}^{n}} & \displaystyle F(x):=f(x)+g(x)\\ \mbox{s.t.} & Ax=b,\\
\end{array}
\end{equation}
where $A\in \mathbb{R}^{m\times n}$ and $b\in\mathbb{R}^{m};$
$f(x)$ is a convex smooth function with {$L_f$}-Lipschitz continuous gradient; and $g(x)$ is a closed convex (not necessarily smooth) function. An important example of problem \eqref{problemcouple} is that $g(x)$ is an indicator function of a closed convex set $\X,$ that is,
\begin{equation*}g(x)=\text{Ind}_{\X}(x):=\left\{
  \begin{array}{cl}
    0,~\text{if}~x\in \X,\\
    +\infty,~\text{otherwise}.
  \end{array}\right.
\end{equation*} In this case, problem \eqref{problemcouple} can be rewritten as
\begin{equation}\label{problemsmooth}
\begin{array}{cl}
\displaystyle \min_{x\in \X}~ \displaystyle f(x),~\mbox{s.t.} \quad Ax=b.
\end{array}
\end{equation}

}

One efficient approach to solving problem \eqref{problemcouple} is the {augmented Lagrangian (AL) method} \cite{hestenes69,powell69,rockafellar}.
The AL function of problem \eqref{problemcouple} is \be\label{augLag} \LCal_\beta(x;\lambda) := \hat f_{\beta}(x;\lambda)+g(x), \ee
where
\begin{equation}\label{hatf}\hat f_{\beta}(x;\lambda):=f(x)+\langle \lambda, Ax-b \rangle + \frac{\beta}{2} \|Ax-b\|^2,\end{equation}
$\lambda\in\mathbb{R}^m$ is the Lagrange multiplier associated with the linear constraint, and $\beta>0$ is {{the penalty parameter}}. The augmented Lagrangian dual 
of problem \eqref{problemcouple} is \begin{equation}\label{dual}
\max_{\lambda\in\mathbb{R}^{m}}\ d(\lambda)
\end{equation}
with \be\label{subproblem}d(\lambda):=\min_{x\in\mathbb{R}^{n}} \LCal_\beta(x;\lambda).\ee
It is well-known that the dual function $d(\lambda)$ in \eqref{subproblem} is differentiable and its gradient is given by $\nabla d(\lambda)=Ax(\lambda)-b,$ where $x(\lambda)$  is the solution of problem \eqref{subproblem} (see \cite{ber99programming}).

Given $\lambda^k,$ the AL method for solving problem \eqref{problemcouple} updates the primal and dual variables via
\begin{equation}\label{primalupdate}
x(\lambda^k)=\displaystyle\arg\min_{x\in\mathbb{R}^{n}} \LCal_\beta(x;\lambda^k)
\end{equation}
and $$\lambda^{k+1}= \lambda^k + \beta\left(Ax(\lambda^k)-b\right),$$ respectively. The AL method for solving problem \eqref{problemcouple} is essentially a dual gradient ascent method {for solving the dual problem \eqref{dual}}, which updates the dual variable by performing a dual gradient ascent step $$\lambda^{k+1}=\lambda^k+ \beta \nabla d(\lambda^k).$$
{The AL method can also be interpreted as a {proximal point algorithm} applied to solve the classical {Lagrangian dual problem \cite{rockafellar} \begin{equation}\label{classicaldual}\max_{\lambda} d'(\lambda),\end{equation}} where $$d'(\lambda):=\min_{x\in\mathbb{R}^{n}} \left\{f(x)+\langle \lambda, Ax-b \rangle +g(x)\right\}.$$}{In this paper, we always refer to the augmented Lagrangian dual (problem) whenever we talk about the Lagrangian dual (problem) or the dual (problem), unless otherwise specified.}
{More generally, the AL method can be derived through the Bregman regularization approach \cite{osher2005iterative,yin2008bregman,goldstein2009split} and it enjoys the so-called error-forgetting property\cite{errorforgetting} when applied to solve problem \eqref{problemcouple} where $F(x)$ is a piece-wise linear function.} {For various variants of the AL method with nonquadratic penalty terms and other multiplier update formulas, please see \cite[Chapter 5]{bertsekas2014constrained}, \cite{iusem1995convergence,tapia,yuan,buys}.}

When the problem dimension $n$ is large, finding an exact solution of AL subproblem \eqref{primalupdate} can be computationally expensive and thus the exact gradient $\nabla d(\lambda^k)$ is often unavailable. As a result, many works focused on inexact versions of (dual) gradient methods; {{see \cite{bertsekas2000gradient,eckstein2013practical,nesterov14inexact,lan2015iteration,luo1993error,necoara15,necoara152,necoara14siamopt,rockafellar,so13inexact,sun15sdpnal+,rockafellar1973multiplier,sun17,tval}}} and references therein. {For instance, Necoara and Patrascu \cite{necoara15} analyzed dual first-order methods for solving a class of \emph{strongly} convex conic programs and provided a detailed (ergodic and non-ergodic) convergence rate analysis of the methods. The methods in \cite{necoara15} are the exact gradient methods applied to solve the dual problem \eqref{dual} where the penalty parameter $\beta$ in \eqref{hatf} is set to be zero ({i.e., problem \eqref{classicaldual}}) and thus is different from the AL method where the penalty parameter $\beta$ in \eqref{hatf} is positive. }
{For the inexact augmented Lagrangian (IAL) framework,} Rockafellar \cite{rockafellar} proposed an IAL framework, where the AL subproblem is solved until a point $x^{k+1}$ is found such that
\begin{equation}\label{neseta}\LCal_\beta\left(x^{k+1};\lambda^k\right)-\LCal_\beta\left(x(\lambda^k);\lambda^k\right)\leq \eta_k^2,\end{equation} and showed that the proposed IAL framework converges if the nonnegative tolerance sequence $\left\{\eta_k\right\}$ is summable.
Very recently, Devolder, Glineur, and Nesterov \cite{nesterov14inexact} proposed a general inexact gradient framework and analyzed the ergodic convergence rate of their framework when it is applied to solve dual problem \eqref{dual}. In \cite{necoara14siamopt}, Nedelcu, Necoara, and Tran{-}Dinh proposed an IAL method, where the AL subproblem was approximately solved by Nesterov's  gradient method {{\cite{nesterov83method,nesterov05mpsmooth,nesterov13composite}}}
such that \eqref{neseta} is satisfied and showed again the ergodic convergence rate of the proposed IAL method. {The non-ergodic convergence rate result for the IAL framework/method has been missing in the literature for a long time until in a very recent work by Lan and Monterio \cite{lan2015iteration}, where they proposed an IAL method (where the AL subproblems are approximately solved by Nesterov's  gradient method) and analyzed the non-ergodic convergence rate for the proposed method.}

{
We make the following assumptions throughout this paper.}

\begin{itemize}
  {\item [A1] there exists a Lagrange multiplier $\lambda^*$ such that the optimal {value} of problem \eqref{problemcouple} is equal to $d(\lambda^*);$
  \item [A2] the function $g(x)$ has a bounded domain. }
\end{itemize}


{Assumption A1 is the strong duality assumption and assumption A2 is made in the paper mainly for the ease of presentation. In fact, for problem \eqref{problemcouple} arising from many applications of interest such as machine learning, statistics, and signal processing, we often can easily find a bounded set $\X$ such that the solution of problem \eqref{problemcouple} lies in $\X.$ Therefore, we can restrict the definition of $g(x)$ over this bounded set. Let us take the following basis pursuit problem in compressed sensing as an example:
\begin{equation}\label{problemcs}
\begin{array}{cl}
\displaystyle \min_{x\in\mathbb{R}^{n}} & \displaystyle \|x\|_1\\ \mbox{s.t.} & Ax=b,\\
\end{array}
\end{equation}
 which is a special case of problem \eqref{problemcouple} with $f(x)=0$ and $g(x)=\|x\|_1.$ We can restrict the definition of $\|x\|_1$ over the bounded domain $$\left\{x\,|\,\|x\|_1\leq \|\hat x\|_1\right\},$$ where $\hat x$ is any point satisfying $A\hat x=b.$} {It is worth remarking that problem \eqref{problemsmooth} has been considered in the existing papers such as \cite{necoara14siamopt,lan2015iteration,necoara152} and they all assumed that the set $\cal X$ {is} convex and compact.}


The contribution of this paper is {twofold}. First, we propose a new IAL framework (see Algorithm \ref{alg:cgal}) for solving problem \eqref{problemcouple}, where the AL subproblem is approximately solved until a point $x^{k+1}$ is found such that
\begin{equation}\label{violationcgal}\max_{x\in\mathbb{R}^{n}} \left\{\left\langle{{\nabla \hat f_{\beta}(x^{k+1};\lambda^k)}},\,x^{k+1}-x\right\rangle + g(x^{k+1})-g(x)\right\}\leq \eta_k.\end{equation} {Here $\nabla \hat f_{\beta}(x;\lambda)$ is the gradient of $\hat f_{\beta}(x;\lambda)$ with respect to $x$.} The termination condition \eqref{violationcgal} in our proposed IAL framework is weaker and (potentially) easier to check than \eqref{neseta} in most of the existing IAL frameworks/methods. {More specifically, {to check whether $x^{k+1}$ satisfies \eqref{violationcgal} or not, we only need to solve the convex optimization problem on the left-hand side of \eqref{violationcgal}, which can be solved exactly or to a high precision in time (essentially) linear to the size of the input for many $g(x)$ such as the $\ell_1$-norm and the nuclear norm; see more examples in \cite{thesiscndg}. In contrast, it is generally hard to check whether $x^{k+1}$ satisfies \eqref{neseta} or not (because $x(\lambda^k)$ is unknown).}} Second, {we establish the global convergence of the proposed IAL framework under the assumption that the sequence $\left\{\eta_k\right\}$ in \eqref{violationcgal} is summable; see Theorem \ref{thm-convergence}. We also show, in Theorems \ref{thm-rate} and \ref{thm-rate-global}, the non-ergodic convergence rate (under weaker conditions than that in \cite{lan2015iteration}) for the proposed IAL framework, which reveals how the error in solving the AL subproblem affects the convergence rate.}

{It is worth
highlighting here that the non-ergodic analysis focuses on the iterates generated by
the algorithm while the ergodic analysis focuses on some (weighted) average of
the iterates generated by the algorithm. In practice, the non-ergodic iterates
tend to share structural properties of the solution of the problem such as sparsity in $\ell_1$ minimization problem \eqref{problemcs}, while the ergodic iterates tend to ``average
out" these properties. Therefore, the non-ergodic solution is more preferable in practical applications. In fact, our simulation results on the basis pursuit problem in Section \ref{sec:simulation} show that the last iterate indeed is much better than the average of all iterates in terms of the sparsity. From the perspective of theoretical analysis, the non-ergodic convergence implies and {thus is} stronger than the ergodic convergence. This paper will focus on the non-ergodic convergence analysis.
}

\section{The IAL framework}\label{sec:cgal}

In this section, we present the IAL framework for solving problem \eqref{problemcouple}. 
The proposed IAL framework is given in Algorithm \ref{alg:cgal}. At the $k$-th iteration, the IAL framework first solves AL subproblem \eqref{subproblemcgal} with fixed dual variable $\lambda^k$ in an inexact manner until a point $x^{k+1}$ satisfying \eqref{violationcgal} is found; then updates the dual variable by performing an inexact gradient ascent step \eqref{lambda}.

\begin{algorithm2e}[ht]
\caption{The IAL framework for {solving} problem \eqref{problemcouple}}
\label{alg:cgal}
\SetKwInOut{Input}{input}\SetKwInOut{Output}{output}
\SetKwComment{Comment}{}{}
\BlankLine
Initialize $x^1 \in \X,~\lambda^1\in\mathbb{R}^m,$ and the nonnegative sequence $\left\{\eta_k\right\}.$ \\
\For{$k\geq 1$:}{

Find an approximate solution $x^{k+1}$ of the AL subproblem \begin{equation}\label{subproblemcgal}\displaystyle \min_{x\in\mathbb{R}^{n}} \left\{\LCal_\beta(x;\lambda^k):=\hat f_{\beta}(x;\lambda^k)+g(x)\right\}\end{equation} such that \eqref{violationcgal} is satisfied;

Update the dual variable via \begin{equation}\label{lambda}
\lambda^{k+1}=\lambda^k+\beta\left(Ax^{k+1}-b\right).
  \end{equation}
}
\end{algorithm2e}

Three remarks on the proposed IAL framework are in order.
{First, the termination condition \eqref{violationcgal} in our proposed IAL framework is (potentially) easier to check than \eqref{neseta} in most of the existing IAL frameworks/methods. Let us take problem \eqref{problemcs} as an example again. In this case, to check whether $x^{k+1}$ satisfies \eqref{violationcgal} or not, we only need to solve the following convex optimization problem
$$\max_{\|x\|_1\leq\|\hat x\|_1} \left\{\left\langle \nabla \hat f_{\beta}(x^{k+1};\lambda^k),\,x^{k+1}-x\right\rangle + \left\|x^{k+1}\right\|_1-\left\|x\right\|_1\right\},$$ where $\nabla \hat f_{\beta}(x^{k+1};\lambda^k)=A^T\left(\lambda^k+\beta\left(Ax^{k+1}-b\right)\right).$ Let $\bar i_{k+1}$ be the index of the largest entry of $\nabla \hat f_{\beta}(x^{k+1};\lambda^k)$ in magnitude, then the solution to the above optimization problem is
\begin{equation*}\bar x^{k+1}=\left\{
\begin{array}{ll}-\|\hat x\|_1{\text{sign}\left(\left[\nabla \hat f_{\beta}(x^{k+1};\lambda^k)\right]_{\bar i_{k+1}}\right)}e_{\bar i_{k+1}},&\mbox{if $\left|\left[\nabla \hat f_{\beta}(x^{k+1};\lambda^k)\right]_{\bar i_{k+1}}\right|\geq 1$;}\\
0,&\mbox{otherwise,}\end{array}\right.
\end{equation*}where $e_{\bar i_{k+1}}$ is the $n$-dimensional vector with the $\bar i_{k+1}$-th entry being $1$ and all other entries being $0$ {and $\text{sign}(\cdot)$ is the sign function}. 
}


Second, the smaller the tolerance $\eta_k$ is, the more computational cost is needed in Algorithm \ref{alg:cgal} to find the point $x^{k+1}$ satisfying \eqref{violationcgal}. On the other hand, the larger the tolerance $\eta_k$ is, the larger the approximation error between the approximate gradient $Ax^{k+1}-b$ and the true gradient $\nabla d(\lambda^k)$ is (see Lemma \ref{lemma-ggap} further ahead), which might lead to slow convergence or even divergence of the proposed Algorithm \ref{alg:cgal}. {Therefore, the choice of $\left\{\eta_k\right\}$ is important in balancing the computational cost (of finding the point $x^{k+1}$ satisfying \eqref{violationcgal}) and the global convergence and convergence rate (of the framework). We will discuss the possible choices of $\left\{\eta_k\right\}$ in more details in Section \ref{subsec:rate}.}

{Third, AL subproblem \eqref{subproblemcgal} can be efficiently solved in an inexact manner by various (first-order) methods such as {the accelerated proximal gradient methods} in {{\cite{nesterov83method,nesterov05mpsmooth,nesterov13composite,beck09siamfista}}} and the Frank-Wolfe (a.k.a. conditional gradient) methods in \cite{fw56,nesterov15nonsmooth,lansco14,Jaggi13,nemirovski14cg}. Next, we discuss the (inner) iteration complexity of finding the point $x^{k+1}$ satisfying \eqref{violationcgal} when {the accelerated proximal gradient methods} and the Frank-Wolfe methods are applied to solve problem \eqref{subproblemcgal}. {Although any of the accelerated proximal gradient methods in \cite{nesterov83method,nesterov05mpsmooth,nesterov13composite,beck09siamfista} and also any of the Frank-Wolfe methods in \cite{fw56,nesterov15nonsmooth,lansco14,Jaggi13,nemirovski14cg} can be used in solving AL subproblem \eqref{subproblemcgal}, we choose the algorithms in \cite{beck09siamfista} and \cite{nesterov15nonsmooth} in our following analysis.}

Let us fist define
  \begin{equation}\label{Lhatf}
    L_{\hat f} = L_f + \beta \|A\|^2,
  \end{equation}where $\|A\|$ denotes the largest singular value of the matrix $A.$ It is simple to see that $\nabla_x \hat f_{\beta}(x;\lambda)$ (with respect to $x$) is Lipschitz continuous with Lipschitz constant $L_{\hat f}.$ Note that $L_{\hat f}$ does not depend on the Lagrange multiplier $\lambda.$ Moreover, let $\cal X$ denote the bounded domain of the function $g(x)$ and let \begin{equation}\label{Dx}D=\max_{x,y\in\cal X} \left\|x-y\right\|<+\infty\end{equation} denote the diameter of the set $\cal X.$

\begin{algorithm2e}[ht]
\caption{{The fast iterative shrinkage-thresholding algorithm (FISTA) {for solving AL subproblem \eqref{subproblemcgal}} \cite{beck09siamfista}}}
\label{alg:nonsmoothCG}
\SetKwInOut{Input}{input}\SetKwInOut{Output}{output}
\SetKwComment{Comment}{}{}
\BlankLine
Initialize $y^{k,1}=x^{k,0}\in \X$ and $t^{1}=1$.\\
\For{$\ell\geq 1$:}{
Set $t^{\ell+1} = \dfrac{1+\sqrt{1+4(t^{\ell})^2}}{2}$\;
Compute $$ x^{k,\ell} = \arg\min_{x\in\mathbb{R}^{n}}\left\{\left\langle \nabla \hat f_{\beta}(y^{k,\ell};\lambda^k),\,x-y^{k,\ell}\right\rangle+\frac{L_{\hat f}}{2}\left\|x - y^{k,\ell}\right\|^2+g(x)\right\}$$
and
$$y^{k, \ell+1} = x^{k, \ell} + \left(\frac{t^{\ell}-1}{t^{\ell+1}}\right) \left(x^{k,\ell} - x^{k,\ell-1}\right);$$
}
\end{algorithm2e}

\begin{theorem}\label{ThmNesInner}
  Let ${\left\{x^{k,\ell}\right\}_{\ell\geq 1}}$ be the sequence generated by {FISTA {(i.e., Algorithm 2)}} when applied to solve AL subproblem \eqref{subproblemcgal}, where $\ell$ is the index of the inner iteration. Suppose that $\hat x^{k,\ell}$ is the point such that
  \begin{equation}\label{proximalgradient}
    \hat x^{k,\ell} = \arg\min_{x\in\mathbb{R}^{n}}\left\{\left\langle \nabla \hat f_{\beta}(x^{k,\ell};\lambda^k),\,x-x^{k,\ell}\right\rangle+\frac{L_{\hat f}}{2}\left\|x - x^{k,\ell}\right\|^2+g(x)\right\}.
  \end{equation}
  Then,
  \begin{equation}\label{ratenes}\max_{x\in\mathbb{R}^{n}}\left\{\left\langle \nabla \hat f_{\beta}(\hat x^{k,\ell};\lambda^k),\,\hat x^{k,\ell} - x\right\rangle + g(\hat x^{k,\ell})-g(x)\right\}\leq \frac{{4L_{\hat f} D^2}}{\ell+1},~\forall~{\ell \geq 1}.\end{equation}
  In other words, it takes at most \begin{equation}\label{numbernes}\left\lceil\frac{{4L_{\hat f}D^2}}{\eta_k}\right\rceil-1\end{equation} {FISTA} iterations and one proximal gradient iteration (equivalent to solving problem \eqref{proximalgradient}) to find the point $x^{k+1}$ satisfying \eqref{violationcgal}.
\end{theorem}
\begin{proof} For simplicity, denote $\hat f_{\beta}(x; \lambda)$ by $\hat f(x),$ $\LCal_{\beta}(x; \lambda)$ by $\LCal(x),$ and $x^{k,\ell}~\text{and}~\hat x^{k,\ell}$ by $x^{\ell}~\text{and}~\hat x^{\ell}$ (respectively) for all ${\ell\geq 1}$ in the proof. First, {it follows from \cite[Theorem 4.4]{beck09siamfista}} and the definitions of $L_{\hat f}$ in \eqref{Lhatf} and $D$ in \eqref{Dx} that
  \begin{equation}\label{nestgrad}
    \LCal(x^{\ell}) -  \LCal(x(\lambda^k)) \leq {\frac{2L_{\hat f} D^2}{(\ell+1)^2}},~\forall~{\ell \geq 1}.
  \end{equation} Note that $\hat x^{\ell}$ in \eqref{proximalgradient} is obtained after performing a proximal gradient step from $x^{\ell}.$ Then, from \cite[Lemma 2.3]{beck09siamfista} and \cite[Theorem 2]{nesterov13composite}, we get
  \begin{equation}\label{proximalstep}
    \LCal(x^{\ell}) -  \LCal(\hat x^{\ell}) \geq \frac{L_{\hat f}}{2}\left\|\hat x^{\ell}-x^{\ell}\right\|^2.
  \end{equation}
  Combining \eqref{nestgrad} and \eqref{proximalstep} yields
  \begin{equation}\label{hatxx}
    \left\|\hat x^{\ell}-x^{\ell}\right\| \leq \frac{{2D}}{\ell+1},~\forall~{\ell\geq 1}.
  \end{equation}
  By the optimality of $\hat x^{\ell}$, we have
  \begin{equation}\label{hatxoptcond}\left\langle \nabla \hat f(x^{\ell}) + L_{\hat f}\left(\hat x^{\ell} - x^{\ell}\right),\,x-\hat x^{\ell} \right\rangle + g(x) - g(\hat x^{\ell})\geq 0, ~\forall~x\in \cal X,\end{equation} where $\cal X$ is the domain of $g(x).$
   Therefore, for {any} $x\in\cal X,$ 
   \begin{align*}
     &~\left\langle \nabla \hat f(\hat x^{\ell}),\,\hat x^{\ell} - x\right\rangle + g(\hat x^{\ell})-g(x)\nonumber\\
     =&~\left\langle \nabla \hat f(\hat x^{\ell}) - \nabla \hat f(x^{\ell}),\,\hat x^{\ell} - x\right\rangle + \left\langle \nabla \hat f(x^{\ell}),\,\hat x^{\ell} - x\right\rangle + g(\hat x^{\ell})-g(x)\nonumber\\
       \leq &~\left\langle \nabla \hat f(\hat x^{\ell}) - \nabla \hat f(x^{\ell}),\,\hat x^{\ell} - x\right\rangle + \left\langle L_{\hat f}\left(\hat x^{\ell} - x^{\ell}\right),\,x-\hat x^{\ell}\right\rangle\\
       \leq &~2 L_{\hat f} D \left\|\hat x^{\ell} - x^{\ell}\right\|\\
       \leq &~{4 L_{\hat f} D^2}/(\ell+1),
   \end{align*} where the first inequality comes from \eqref{hatxoptcond}; the second inequality is due to the Cauchy-Schwarz inequality, the definition of $D$ in \eqref{Dx}, and the fact that $\nabla \hat f(x)$ (with respect to $x$) is $L_{\hat f}-$Lipschitz continuous; the third inequality follows from \eqref{hatxx}. Taking the maximum over $x\in\cal X$ in the above inequality leads to the desired result \eqref{ratenes}. \hfill$\Box$
\end{proof}

\begin{algorithm2e}[ht]
\caption{The Frank-Wolfe algorithm {for solving AL subproblem \eqref{subproblemcgal}} \cite{nesterov15nonsmooth}}
\label{alg:nonsmoothCG}
\SetKwInOut{Input}{input}\SetKwInOut{Output}{output}
\SetKwComment{Comment}{}{}
\BlankLine
Initialize ${x^{k,0}\in \X}$.\\
\For{${\ell\geq 0}$:}{
Set $\gamma^{\ell} = 2/(\ell+2)$\;
Compute \begin{equation}\label{vl}\upsilon^{\ell} = \arg\min_{x\in\mathbb{R}^{n}}\left\{\left\langle \nabla \hat f_{\beta}(x^{k,\ell};\lambda^k),\,x-x^{k,\ell}\right\rangle+g(x)\right\};\end{equation}\label{line:vk}
Update
$x^{k, \ell+1} = x^{k, \ell} + \gamma^{\ell} (\upsilon^{\ell} - x^{k,\ell})$\;
}
\end{algorithm2e}

\begin{theorem}\label{ThmFWInner}
  Let $\left\{x^{k,\ell}\right\}_{\ell\geq 1}$ be the sequence generated by the Frank-Wolfe method (i.e., Algorithm 3) when applied to solve AL subproblem \eqref{subproblemcgal}, where $\ell$ is the index of the inner iteration. 
  Then there exists $1\leq \hat \ell \leq \ell$ such that
  \begin{equation}\label{ratefw}\max_{x\in\mathbb{R}^{n}}\left\{\left\langle \nabla \hat f_{\beta}(x^{k,\hat \ell};\lambda^k),\,x^{k,\hat \ell} - x\right\rangle + g(x^{k,\hat \ell})-g(x)\right\}\leq \frac{6L_{\hat f} D^2}{\ell + 2},~\forall~\ell \geq 1.\end{equation}
  In other words, it takes at most \begin{equation}\label{numberfw}\left\lceil\frac{6L_{\hat f}D^2}{\eta_k}\right\rceil{-2}\end{equation} Frank-Wolfe iterations to find the point $x^{k+1}$ satisfying \eqref{violationcgal}.
\end{theorem}
\begin{proof}
For simplicity, denote $\hat f_{\beta}(x; \lambda)$ by $\hat f(x),$ $\LCal_{\beta}(x; \lambda)$ by $\LCal(x),$ and $x^{k,\ell}$ by $x^{\ell}$ for all ${\ell\geq 0}$ in the proof. {Define $$\Delta^{\ell}=\LCal(x^{\ell})-\LCal(x(\lambda^k))$$ and
\begin{equation*}\label{Vk}
  V^{\ell}=\left\langle\nabla \hat f(x^{\ell}),\,x^{\ell}-\upsilon^{\ell}\right\rangle+g(x^{\ell})-g(\upsilon^{\ell}).
\end{equation*}}From the convexity of $\hat f(x)$ and the definition of $\upsilon^{\ell}$ in \eqref{vl}, we get
\begin{equation}\label{FV}\Delta^{\ell}\leq \left\langle\nabla \hat f(x^{\ell}),\,x^{\ell} - x(\lambda^k)\right\rangle +g(x^{\ell}) - g(x(\lambda^k)) \leq V^{\ell}.\end{equation}
 {Recall $x^{\ell+1}=x^{\ell} + \gamma^{\ell} (\upsilon^{\ell} - x^{\ell}).$} Since $\nabla \hat f(x)$ is $L_{\hat f}$-Lipschitz continuous, it follows that
\begin{equation}\label{ineqfLip}
  \hat f(x^{\ell+1})\leq \hat f(x^{\ell})+\gamma^{\ell}\left\langle\nabla \hat f(x^{\ell}), \upsilon^{\ell} - x^{\ell}\right\rangle +\frac{L_{\hat f}\left(\gamma^{\ell}\right)^2}{2}\left\|\upsilon^{\ell}-x^{\ell}\right\|^2.
\end{equation}
By the convexity of $g(x),$ we get 
\begin{equation}\label{ineqgconvex}
  g(x^{\ell+1})\leq \gamma^{\ell} g(\upsilon^{\ell})+\left(1-\gamma^{\ell}\right) g(x^{\ell}).
\end{equation}
Combining \eqref{ineqfLip} and \eqref{ineqgconvex}, 
we have, for ${\ell=0,1,...,}$
\begin{align}\label{ineqhatF}
  \LCal(x^{\ell+1})\leq \LCal(x^{\ell})-\gamma^{\ell} V^{\ell}+\frac{L_{\hat f}\left(\gamma^{\ell}\right)^2}{2}\left\|\upsilon^{\ell}-x^{\ell}\right\|^2,
\end{align}
which, together with \eqref{FV}, implies
\begin{equation*}
\Delta^{\ell+1}\leq \left(1-\gamma^{\ell}\right)\Delta^{\ell}+\frac{L_{\hat f}\left(\gamma^{\ell}\right)^2}{2}\left\|\upsilon^\ell-x^\ell\right\|^2
\end{equation*}
 and thus
 \begin{equation*}\label{decreasehatF}\Delta^{\ell+1}\leq \frac{\ell}{\ell+2}\Delta^{\ell}+\left(\frac{2}{\ell+2}\right)^2\frac{L_{\hat f}D^2}{2}.\end{equation*}
By mathematical induction, it can be verified that
\begin{equation}\label{cgfunction}{\Delta^{\ell}}\leq \frac{2L_{\hat f}D^2}{\ell+2},~\ell\geq 1.\end{equation}

{Based on \eqref{cgfunction}, we can use the same argument (by contradiction) as in \cite{Jaggi13,mu14matrixrecovery} to show the desired result \eqref{ratefw}. The idea is to show that
$\left\{V^m\right\}_{m=1}^\ell$ cannot stay large over many consecutive iterations. For completeness, we give the details below. Denote
$$\hat C=2L_{\hat f}D^2,~\ell_1=\lceil\frac{\ell}{2}\rceil,~\text{and}~\nu=\frac{\ell_1+1}{\ell+2}.$$ Suppose on the contrary that
\begin{equation}\label{contrary}
  V^{m} > \frac{3\hat C}{\ell+2}, \forall~m = \ell_1, \ell_1+1,\ldots, \ell.
\end{equation} From \eqref{ineqhatF}, we have, 
$$\Delta^{\ell+1} \leq \Delta^{\ell} - \frac{2V^{\ell}}{\ell+2} + \frac{\hat C}{\left(\ell+2\right)^2},~{\ell\geq0}.$$
Summing this inequality for indices from $\ell_1$ to $\ell$ yields
  \begin{eqnarray}
    \Delta^{\ell+1} &\leq & \Delta^{\ell_1}- \sum_{m=\ell_1}^{\ell} \frac{2V^{m}}{m+2} + \sum_{m=\ell_1}^{\ell} \frac{\hat C}{\left(m+2\right)^2} \nonumber \\
                    &< & \Delta^{\ell_1}- \frac{6\hat C}{\ell+2}\sum_{m=\ell_1+2}^{\ell+2} \frac{1}{m} + \sum_{m=\ell_1+2}^{\ell+2} \frac{\hat C}{m^2} \nonumber \\
                    &\leq & \frac{\hat C}{\nu\left(\ell+2\right)} - \frac{6\hat C}{\ell+2} \frac{\ell-\ell_1+1}{\ell+2} + \frac{\hat C\left(\ell-\ell_1+1\right)}{\left(\ell+2\right)\left(\ell_1+1\right)} \nonumber \\
                    & = & \frac{\hat C}{\nu\left(\ell+2\right)}\left(2-6\nu\left(1-\nu\right)-\nu\right), \label{ineqcontrary}
  \end{eqnarray}
where the second inequality is due to \eqref{contrary}, and the third inequality is due to \eqref{cgfunction} and the fact $\sum_{m=a}^b\frac{1}{k^2}\leq \frac{b-a+1}{b(a-1)}$ for any $b\geq a >1.$ Define $\phi(x)=2-6x\left(1-x\right)-x.$ Since $\nu\in[\frac{1}{2}, \frac{2}{3}],$ it follows from \eqref{ineqcontrary} that $\Delta^{\ell+1}<\frac{\hat C}{\nu\left(\ell+2\right)}\phi(\nu)\leq 0,$ which is a contradiction. The proof is completed.} \hfill$\Box$
\end{proof}

The convergence rate result \eqref{cgfunction} is not new when $g(x)=\text{Ind}_{\X}(x)$; see \cite{Jaggi13,lan14cgs,mu14matrixrecovery} and references therein. For the general composite minimization case, Nesterov proved $${\LCal(x^{\ell})}-\LCal(x(\lambda^k))\leq \frac{4L_{\hat f}D^2}{{\ell+1}},~\ell\geq 1;$$ {see Eq. (2.16) in \cite{nesterov15nonsmooth}}. {Roughly speaking}, our result \eqref{cgfunction} improves the above bound by a factor of two. 

Compared to {accelerated proximal gradient methods} when applied to solve AL subproblem \eqref{subproblemcgal}, the Frank-Wolfe methods generally need {more iterations} to find the point $x^{k+1}$ satisfying \eqref{violationcgal}, while the computational cost per iteration in the Frank-Wolfe methods is generally cheaper.}
\section{Global convergence}\label{subsec:convergence}

{In this {section}, we present the global convergence result of our IAL framework (Algorithm \ref{alg:cgal}), which {is independent of} the methods used to find the point $x^{k+1}$ satisfying \eqref{violationcgal}.} Theorem \ref{thm-convergence} shows global convergence of the IAL framework under the assumption that the nonnegative sequence $\left\{\eta_k\right\}$ is summable.

{It is worth remarking that \eqref{upper_bound_dual}, \eqref{decrease}, and \eqref{gradient}, which build a bridge between the exact dual function value and dual gradient and the approximate ones, are crucial for establishing global convergence and non-ergodic convergence rate results in this paper. We shall show that condition \eqref{violationcgal} {implies} \eqref{upper_bound_dual}, \eqref{decrease}, and \eqref{gradient}; see the proofs of Lemma \ref{lemma_inequality} and Lemma \ref{lemma-ggap}. Clearly, condition \eqref{violationcgal} can be replaced with some other conditions (e.g., condition \eqref{neseta}) in Algorithm \ref{alg:cgal} and global convergence and non-ergodic convergence results of Algorithm \ref{alg:cgal} will still follow as long as the new conditions imply \eqref{upper_bound_dual}, \eqref{decrease}, and \eqref{gradient}. }

%


For the ease of presentation, we define
\begin{align}
  x(\lambda^k):=&\arg\min_{x\in\mathbb{R}^{n}} \LCal_\beta(x;\lambda^k),~k=1,2,\ldots,\label{xlambda}\\
  \nabla d(\lambda^k):=&Ax(\lambda^{k})-b,~k=1,2,\ldots,\nonumber\\
  \bar d(\lambda^k):=&\LCal_{\beta}(x^{k+1};\lambda^k),~k=1,2,\ldots,\label{bard}\\
  \nabla \bar d(\lambda^k):=&Ax^{k+1}-b,~k=1,2,\ldots,\label{bardgradient}
\end{align}
%
%
where $x^{k+1}$ is generated by Algorithm \ref{alg:cgal} and satisfies \eqref{violationcgal}. 

\medskip
We first prove the following two lemmas (Lemma \ref{lemma_inequality} and Lemma \ref{lemma-ggap}), which have been proved for smooth function $F(x)$ in \cite{nesterov14inexact,necoara14siamopt,lan2015iteration}. We now extend them to {composite} nonsmooth function $F(x).$ In particular, Lemma \ref{lemma_inequality} shows that $d(\lambda^{k+1})$ can be bounded from both above and below and Lemma \ref{lemma-ggap} shows that $\left\|\nabla d(\lambda^k)-\nabla\bar d(\lambda^k)\right\|$ is bounded by $\sqrt{\eta_k/\beta}.$

\begin{lemma}\label{lemma_inequality}
The following two inequalities hold:
{ \begin{equation}\label{upper_bound_dual}
  d(\lambda) \leq \bar d(\mu) + \left\langle\nabla\bar d(\mu),\lambda-\mu\right\rangle,~\forall~\lambda,\mu,
\end{equation}}
and
{\begin{equation}\label{decrease}d(\lambda^{k+1})\geq\bar d(\lambda^k)+\frac{\beta}{2}\left\|\nabla \bar d(\lambda^k)\right\|^2-\eta_k.\end{equation}}
\end{lemma}
\begin{proof} We first show \eqref{upper_bound_dual}.
{By the definitions of $d(\lambda)$ and $\LCal_{\beta}(x;\lambda)$, we have
$$d(\lambda)=\min_{x\in\mathbb{R}^{n}}\left\{\LCal_{\beta}(x;\lambda)\right\}\leq \LCal_{\beta}(x_{\mu};\lambda)=\LCal_{\beta}(x_{\mu};\mu) + \langle\lambda-\mu,Ax_{\mu}-b\rangle,$$ where $x_{\mu}$ satisfies $\bar d(\mu)=\LCal_{\beta}(x_{\mu};\mu).$
This, together with the definition of $\nabla \bar d(\mu)$ (cf. \eqref{bard}), yields \eqref{upper_bound_dual}.}

{We now prove \eqref{decrease}. By the convexity of $f(x),$ the definition of $\nabla \bar d(\lambda^k),$ and \eqref{lambda}, we get
 \begin{align*}
   \LCal_{\beta}(x;\lambda^{k+1}) \geq  &\  f(x^{k+1})+\left\langle\nabla f(x^{k+1}),x-x^{k+1}\right\rangle+g(x)+\left\langle \lambda^{k+1}, Ax-b \right\rangle + \frac{\beta}{2} \left\|Ax-b\right\|^2\\
                       = &\  \LCal_{\beta}(x^{k+1};\lambda^{k})+\beta\left(\left\langle\nabla \bar d(\lambda^k), Ax-b\right\rangle+\frac{1}{2}\left\|\left(Ax-b\right)-\nabla \bar d(\lambda^k)\right\|^2\right)\\
                         &\  +\left\langle\nabla \hat f_{\beta}(x^{k+1};\lambda^k),x-x^{k+1}\right\rangle+g(x)-g(x^{k+1}).
 \end{align*}
 Taking the minimum over $x\in\mathbb{R}^{n}$ on both sides of the above inequality, we have
  \begin{align*}
   d(\lambda^{k+1}) \geq  &\  \LCal_{\beta}(x^{k+1};\lambda^{k})+\beta\min_{x\in\mathbb{R}^{n}}\left\{\left\langle\nabla \bar d(\lambda^k), Ax-b\right\rangle+\frac{1}{2}\left\|\left(Ax-b\right)-\nabla \bar d(\lambda^k)\right\|^2\right\}\\
                         & +\min_{x\in\mathbb{R}^{n}}\left\{\left\langle\nabla \hat f_{\beta}(x^{k+1};\lambda^k),x-x^{k+1}\right\rangle+g(x)-g(x^{k+1})\right\}.
 \end{align*}
By using the definition of $\bar d(\lambda)$ and \eqref{violationcgal}, we immediately get the desired result \eqref{decrease}. \hfill$\Box$}
\end{proof}

\begin{lemma}\label{lemma-ggap}
The following inequality holds:
\begin{equation}\label{gradient}
  \left\|\nabla d(\lambda^k)-\nabla\bar d(\lambda^k)\right\|^2\leq \frac{\eta_k}{\beta}.
\end{equation}
\end{lemma}\begin{proof} It follows from the optimality of $x(\lambda^k)$ (cf. \eqref{xlambda}) that
$$\left\langle\nabla\hat f_{\beta}(x(\lambda^k);\lambda^k),x(\lambda^k)-x^{k+1}\right\rangle+g(x(\lambda^k))-g(x^{k+1})\leq 0.$$ By setting $x=x(\lambda^k)$ in \eqref{violationcgal}, we get $$\left\langle \nabla\hat f_{\beta}(x^{k+1};\lambda^k),x^{k+1}-x(\lambda^k)\right\rangle-g(x(\lambda^k))+g(x^{k+1})\leq \eta_k.$$
Adding the above two inequalities yields
\begin{align}
  \eta_k\geq & \left\langle \nabla\hat f_{\beta}(x^{k+1};\lambda^k)-\nabla\hat f_{\beta}(x(\lambda^k);\lambda^k),x^{k+1}-x(\lambda^k)\right\rangle\nonumber\\
         =   & \left\langle \nabla f(x^{k+1})+ \beta A^T(Ax^{k+1}-b)-\nabla f(x(\lambda^k)) - \beta A^T(Ax(\lambda^k)-b) ,x^{k+1}-x(\lambda^k)\right\rangle\nonumber\\
\geq   & \left\langle \beta A^T(Ax^{k+1}-b) - \beta A^T(Ax(\lambda^k)-b) ,x^{k+1}-x(\lambda^k)\right\rangle\nonumber\\
= & \beta\left\|\nabla d(\lambda^k)-\nabla\bar d(\lambda^k)\right\|^2,\nonumber
\end{align}where the second inequality is due to the convexity of $f(x)$ and the second equality is due to the definitions of $\nabla d(\lambda^k)$ and $\nabla\bar d(\lambda^k).$ \hfill$\Box$ 
\end{proof}
The following Lemma \ref{lemma:lambda} shows that the sequence $\left\{\lambda^k\right\}$ generated by Algorithm \ref{alg:cgal} is bounded.
\begin{lemma}\label{lemma:lambda}Let $\left\{\lambda^k\right\}$ be generated by Algorithm \ref{alg:cgal}. Suppose the sequence $\left\{\eta_k\right\}$ satisfies \eqref{summable}, then {\begin{equation}\label{boundlambda}\left\|\lambda^k-\lambda^*\right\|\leq B,~k=1,2,\ldots,\end{equation}where 
\begin{equation}\label{B}
  B:=\sqrt{\left\|\lambda^1-\lambda^*\right\|^2+{2\beta}\sum_{k=1}^{+\infty}\eta_k}<+\infty.
\end{equation}
}
\end{lemma}
\begin{proof}
We have
{\begin{align}
  \left\|\lambda^{k+1}-\lambda^*\right\|^2=&\left\|\lambda^k-\lambda^*+{\beta}\nabla\bar d(\lambda^k)\right\|^2\nonumber\\
                               =&\left\|\lambda^k-\lambda^*\right\|^2+{2\beta}\left\langle\nabla\bar d(\lambda^k), \lambda^k-\lambda^*\right\rangle+{\beta^2}\left\|\nabla\bar d(\lambda^k)\right\|^2\nonumber\\
                               \leq & \left\|\lambda^k-\lambda^*\right\|^2+{2\beta}\left(\bar d(\lambda^k)-d(\lambda^*)\right)+{\beta^2}\left\|\nabla\bar d(\lambda^k)\right\|^2\nonumber\\
                               = & \left\|\lambda^k-\lambda^*\right\|^2+{2\beta}\left(d(\lambda^{k+1})-d(\lambda^*)\right)+{2\beta}\left(\bar d(\lambda^k)-d(\lambda^{k+1})\right) +{\beta^2}\left\|\nabla\bar d(\lambda^k)\right\|^2\nonumber\\
                               \leq & \left\|\lambda^k-\lambda^*\right\|^2+{2\beta}\left(d(\lambda^{k+1})-d(\lambda^*)\right)+{2\beta}\eta_k\nonumber\\
                               \leq & \left\|\lambda^k-\lambda^*\right\|^2+{2\beta}\eta_k,\nonumber
\end{align}}where the first inequality is due to \eqref{upper_bound_dual} (with $\lambda$ and $\mu$ replaced by $\lambda^*$ and $\lambda^k$ respectively), the second inequality is due to \eqref{decrease}, and the last inequality is due to the fact that $d(\lambda^{k+1})\leq d(\lambda^*)$ for all $k\geq 1.$ Summing the above inequality, we obtain
$$\left\|\lambda^k-\lambda^*\right\|^2\leq \left\|\lambda^1-\lambda^*\right\|^2+{2\beta}\sum_{i=1}^{k-1}\eta_i,~k=1,2,\ldots,$$
{which, {together with \eqref{B}}, completes the proof.} \hfill$\Box$
\end{proof}

{Before presenting the main result of this section, i.e., the global convergence result of our IAL framework, we define}
\begin{equation}\label{theta}
    \theta:=\frac{\beta}{4B^2},
\end{equation}where $B$ is given in \eqref{B}.

\begin{theorem}\label{thm-convergence} Let $\left\{x^k\right\}$ and $\left\{\lambda^k\right\}$ be generated by Algorithm \ref{alg:cgal}. Suppose the nonnegative sequence $\left\{\eta_k\right\}$ satisfies \begin{equation}\label{summable}
    \sum_{k=1}^{+\infty}\eta_k<+\infty.
  \end{equation} Then, $$\delta_k:=d\left(\lambda^*\right)-d(\lambda^k)\rightarrow 0~\text{and}~\left\|Ax^{k+1}-b\right\|\rightarrow 0,$$ where {$\lambda^*$ is an optimal solution to problem \eqref{dual}} and $d(\lambda)$ is defined in \eqref{subproblem}.
\end{theorem}
\begin{proof}
  It suffices to show \begin{equation}\label{deltasummable}\sum_{k} \delta_k^2<+\infty,\end{equation} and \begin{equation}\label{axksummable}\sum_{k} \left\|Ax^{k+1}-b\right\|^2<+\infty.
 \end{equation} 

By \eqref{decrease} and the definition of $\bar d(\lambda^k),$ we obtain
\begin{equation}\label{dual_descent}d(\lambda^{k+1})\geq d(\lambda^k)+\frac{\beta}{2}\left\|\nabla \bar d(\lambda^k)\right\|^2-\eta_k,\end{equation}
which, together with \eqref{gradient} and the inequality $a^2\geq b^2/2-(a-b)^2,$ implies
\begin{equation}\label{key_descent}d(\lambda^{k+1})\geq d(\lambda^k)+\frac{\beta}{4}\left\|\nabla d(\lambda^k)\right\|^2-\frac{3}{2}\eta_k.\end{equation}
Moreover, it follows from \eqref{boundlambda} and the concavity of $d(\lambda)$ that
$$d(\lambda^*)-d(\lambda^k)\leq \left\langle\nabla d(\lambda^k), \lambda^*-\lambda^k\right\rangle\leq \left\|\lambda^k-\lambda^*\right\|\left\|\nabla d(\lambda^k)\right\|\leq B\left\|\nabla d(\lambda^k)\right\|.$$
Combining the above, {\eqref{theta}, and \eqref{key_descent}}, we immediately obtain \begin{equation}\label{delta}\delta_{k+1}\leq \delta_k-\theta\delta_k^2+\frac{3}{2}\eta_k,~k=1,2,\ldots,\end{equation} which further implies 
\begin{equation}\label{deltak}\delta_k\leq \delta_1 - \theta\sum_{i=1}^{k-1} \delta_i^2+\frac{3}{2}\sum_{i=1}^{k-1}\eta_i,~k=1,2,\ldots\end{equation}
From the definition of $\delta_k$, we know $\delta_k\geq 0$ for all $k\geq 1.$ From this, \eqref{summable}, and \eqref{deltak}, we obtain \eqref{deltasummable}.

Next, we prove \eqref{axksummable}. It follows from \eqref{bardgradient} and \eqref{dual_descent} that
  \begin{equation}\label{axkineq}\left\|Ax^{k+1}-b\right\|^2\leq  \frac{2}{{\beta}}\left(d(\lambda^{k+1})- d(\lambda^k)+\eta_k\right),~k=1,2,\dots.\end{equation}
Summing \eqref{axkineq} from $i=1$ to $k$ yields
  \begin{align*}
    \sum_{i=1}^k\left\|Ax^{i+1}-b\right\|^2 &\leq  \frac{2}{{\beta}}\left(d(\lambda^{k+1})- d(\lambda^1)+\sum_{i=1}^k\eta_i\right)\\
                                            &\leq  \frac{2}{{\beta}}\left(d(\lambda^{*})- d(\lambda^1)+\sum_{i=1}^k\eta_i\right)\\
                                            &\leq  \frac{2}{{\beta}}\left(\frac{1}{2\beta}\left\|\lambda^1-\lambda^*\right\|^2+\sum_{i=1}^k\eta_i\right)\\
                                            &\leq  \frac{B^2}{{{\beta^2}}},
  \end{align*}where {the third inequality is due to the facts that $\nabla d(\lambda^*)=0$ and $\nabla d(\lambda)$ is $\frac{1}{\beta}$-Lipschitz continuous \cite{ber99programming}} and
  the last inequality is {due to \eqref{B}}. The proof of Theorem \ref{thm-convergence} is completed.\hfill$\Box$
\end{proof}

Theorem \ref{thm-convergence} shows the global convergence of Algorithm \ref{alg:cgal} under conditions \eqref{violationcgal} and \eqref{summable}. {Classical} conditions in \cite{rockafellar} that guarantee the global convergence of the IAL framework are \eqref{neseta} and \eqref{summable}. Since \eqref{neseta} implies \eqref{violationcgal} ({by Theorem \ref{ThmNesInner}}), our conditions \eqref{violationcgal} and \eqref{summable} are weaker than conditions \eqref{neseta} and \eqref{summable} in \cite{rockafellar}. 

\section{Non-ergodic convergence rate}\label{subsec:rate}
In this {section}, we present the non-ergodic convergence rate result of our IAL framework (Algorithm \ref{alg:cgal}). Theorems \ref{thm-rate} and \ref{thm-rate-global} show the non-ergodic convergence rate of the IAL framework.

Since $\delta_k\rightarrow 0$ (cf. Theorem \ref{thm-convergence}) and $\eta_k\rightarrow 0$, there exists $k_0\geq 4$ such that
\begin{equation}\label{k0}
\max_{k\geq k_0}\left\{\delta_{k}\right\} \leq \frac{1}{2\theta}~\text{and}~\max_{k\geq k_0}\left\{\eta_{k}\right\} \leq \frac{1}{24\theta},
\end{equation}{where $\theta$ is given in \eqref{theta}.}
{Define}
\begin{equation}\label{tau}\tau_1:=\frac{k_0}{4\theta}~\text{and}~\tau_2:=\frac{1}{4\theta\sqrt{\eta_{k_0}}}.\end{equation}
It is easy to verify
\begin{equation}\label{tau12great}\tau_1\geq \frac{1}{\theta}~\text{and}~\tau_2\geq\sqrt{\frac{3}{2\theta}}.\end{equation}


\begin{theorem}\label{thm-rate}
Let $\left\{\lambda^k\right\}$ be generated by Algorithm \ref{alg:cgal}. Suppose that the positive sequence $\left\{\eta_k\right\}$ is nonincreasing and satisfies \eqref{summable} and
\begin{equation}\label{etak}
    \sqrt{\frac{\eta_{k+1}}{\eta_k}}\geq \displaystyle \frac{k-2}{k},~k=k_0,\,k_0+1,\ldots,
\end{equation}
where $k_0$ satisfies \eqref{k0}. Then,
\begin{equation}\label{deltaeta}
    \delta_k\leq \frac{\tau_1}{k}+\tau_2\sqrt{\eta_{k}},~k=k_0,\,k_0+1,\ldots,
\end{equation}
where $\tau_1$ and $\tau_2$ are defined in \eqref{tau}.
\end{theorem}
\begin{proof}
We prove the theorem by induction. From \eqref{k0} and \eqref{tau}, we know $$\delta_{k_0}\leq \frac{1}{2\theta}=\frac{\tau_1}{k_0}+\tau_2\sqrt{\eta_{k_0}}.$$ Therefore, the inequality \eqref{deltaeta} holds for $k=k_0$. Next, we assume that \eqref{deltaeta} holds for some $k\geq k_0$, and we consider the
case $k+1$. We have
\begin{align*}
  \delta_{k+1} \leq &\  \delta_k-\theta\delta_k^2+\frac{3}{2}\eta_k\\
               \leq &\  \frac{\tau_1}{k}+\tau_2\sqrt{\eta_k}-\theta\left(\frac{\tau_1}{k}+\tau_2\sqrt{\eta_k}\right)^2+\frac{3}{2}\eta_k\\
               =    &\  \frac{\tau_1}{k+1}\frac{\left(k+1\right)\left(k-\theta\tau_1\right)}{k^2} +\tau_2\sqrt{\eta_{k+1}}\frac{\sqrt{\eta_k}\left(1-\frac{2\theta\tau_1}{k}\right)}{\sqrt{\eta_{k+1}}}+\left(\frac{3}{2}-\theta\tau_2^2\right)\eta_k\nonumber\\
               \leq &\  \frac{\tau_1}{k+1}+\tau_2\sqrt{\eta_{k+1}},
\end{align*}
where the first inequality is due to \eqref{delta}, {the second inequality is due to {the facts that the function $x-\theta x^2$ is an increasing function in $x\in(-\infty, \frac{1}{2\theta}]$ and $\frac{\tau_1}{k}+\tau_2\sqrt{\eta_k}\leq \frac{\tau_1}{k_0}+\tau_2\sqrt{\eta_{k_0}}=\frac{1}{2\theta}$~for all $k\geq k_0$ (because $\left\{\eta_k\right\}$ is nonincreasing),}} and the last inequality is due to \eqref{tau12great} and \eqref{etak}. The proof of Theorem \ref{thm-rate} is completed. \hfill$\Box$
\end{proof}

As shown in \eqref{deltaeta}, the rate that $\left\{\delta_k\right\}$ converges to zero depends on two terms, i.e.,  $\frac{\tau_1}{k}~\text{and}~\tau_2\sqrt{\eta_{k}},$ and the rate is determined by the slower one of them: if $k\sqrt{\eta_{k}}\rightarrow 0$, then $\delta_k\rightarrow 0$ with a rate of ${\cal O}\left(1/k\right);$ otherwise, $\delta_k\rightarrow 0$ with a rate of ${\cal O}\left(\sqrt{\eta_{k}}\right).$
In particular, if $\eta_k=0$ for all $k\geq 1,$ then Algorithm \ref{alg:cgal} reduces to the exact dual gradient ascent method, and achieves the ${\cal O}\left(1/k\right)$ convergence rate.

These facts indicate that the sequence $\left\{\eta_k\right\}$ in Algorithm \ref{alg:cgal} should not be chosen such that $\left\{\sqrt{\eta_k}\right\}$ converges faster than $\left\{1/k\right\}$ to zero. This is {because} such a choice would increase the computational cost of solving the AL subproblem, but theoretically cannot improve the convergence rate of $\left\{\delta_k\right\},$ which is ${\cal O}\left(1/k\right)$ in this case. One possible choice of the sequence $\left\{\eta_k\right\}$ is
\begin{equation}\label{etakalpha}\eta_k=\frac{\sigma}{k^{2\alpha}},~k=1,2,\ldots\end{equation}
with some constant $\sigma>0$ and $\alpha\in(\frac{1}{2},1].$ It is easy to check that \eqref{etakalpha} satisfies all conditions required in Theorem \ref{thm-rate}.

The following Theorem \ref{thm-rate-global} gives the non-ergodic convergence rate of Algorithm \ref{alg:cgal} when ${\eta_k}$ is chosen as in \eqref{etakalpha}. We first present a lemma, which is useful in proving Theorem \ref{thm-rate-global}.

\begin{lemma}\label{prop:1}
Suppose the nonnegative sequence $\left\{\delta_k\right\}$ satisfies
\begin{eqnarray}\label{eq:ineqnonoise2}
\frac{E}{2} \delta_{k+1}^2 + \delta_{k+1} \leq \delta_k,~k=1,2,\ldots,
\end{eqnarray}where $E>0$ is a constant.
Then, we have
\begin{eqnarray}\label{eq:complexity1}
\delta_{k} \leq \frac{\max\left\{\delta_1,\frac{4}{E}\right\}}{k},~k=1,2,\ldots.
\end{eqnarray}
\end{lemma}
\begin{proof}
{Again we prove this by induction.} Clearly, the inequality \eqref{eq:complexity1} is true for $k=1$. Next, assuming \eqref{eq:complexity1} is true for some $k\geq 1$, we show it is also true for $k+1$. In fact, we have
\begin{align*}
\delta_{k+1} \leq \frac{-1+\sqrt{1+2E\delta_k}}{E} \leq &\  \frac{-1+\sqrt{1+2E\frac{\max\left\{\delta_1,\frac{4}{E}\right\}}{k}}}{E}\\=&\ \frac{2\max\left\{\delta_1,\frac{4}{E}\right\}}{k+\sqrt{k^2+2E\max\left\{\delta_1,\frac{4}{E}\right\}k}}\\
\leq&\  \frac{\max\left\{\delta_1,\frac{4}{E}\right\}}{k+1},
\end{align*}
where the first inequality is due to the inequality \eqref{eq:ineqnonoise2}, and the second inequality is due to the assumption that \eqref{eq:complexity1} holds for $k.$\hfill$\Box$
\end{proof}

\begin{theorem}\label{thm-rate-global}
Let $\left\{x^k\right\}$ and $\left\{\lambda^k\right\}$ be generated by Algorithm \ref{alg:cgal}. Suppose that the positive sequence $\left\{\eta_k\right\}$ is chosen as in \eqref{etakalpha}. Then,
 \begin{equation}\label{deltaeta-global}
  \delta_k\leq \frac{C}{k^{\alpha}},~k=1,2,\ldots,
\end{equation}
where \begin{equation}\label{C}C=4\sqrt{\frac{3\left(\frac{3}{2}\theta \sigma+1\right)\sigma}{\theta}}+\frac{4}{3}\max\left\{\delta_1,\frac{4}{\theta}\right\},\end{equation} and $\theta$ is given in \eqref{theta}; \begin{equation}\label{axketa}
    \left\|Ax^{k+1}-b\right\|^2\leq \psi_k:=\frac{2}{\beta}\left(\frac{{C+{\sigma}}}{k^{\alpha}}\right),~k=1,2,\ldots;
  \end{equation}
 \begin{align}\label{primalopt}
 -\left\|\lambda^*\right\|\sqrt{\psi_k}-\frac{\beta}{2}\psi_k\leq F(x^{k+1})-F(x^*)\leq\left(\left\|\lambda^*\right\|+B\right)\sqrt{\psi_k}+\eta_k,~k=1,2,\ldots,
  \end{align}
  where $\lambda^*$ is an optimal solution to problem \eqref{dual} and $B$ is given in \eqref{theta};
  {and
  \begin{equation}\label{Lagrange}\max_{x\in\mathbb{R}^{n}} \left\{\left\langle{{\nabla f(x^{k+1})}+A^T\lambda^{k+1}},\,x^{k+1}-x\right\rangle + g(x^{k+1})-g(x)\right\}\leq \eta_k,~k=1,2,\ldots\end{equation}}

\end{theorem}

\begin{proof}
  We show \eqref{deltaeta-global}, \eqref{axketa}, \eqref{primalopt}, and \eqref{Lagrange} separately.

We first show \eqref{deltaeta-global} by induction. Clearly, the inequality \eqref{deltaeta-global} holds for $k=1$.
Next, we assume that \eqref{deltaeta-global} holds for some $k\geq 1$, and show it is also true for $k+1$. We use the contrapositive argument. Assume that \eqref{deltaeta-global} does not hold for $k+1$, i.e.,
\begin{eqnarray}\label{eq:nothold}
{\delta_{k+1} > \frac{C}{(k+1)^{\alpha}}},
\end{eqnarray}where $C$ is given in \eqref{C}. Let $z^*=\frac{C}{4}-\frac{1}{\theta}>0.$ If \begin{eqnarray} \label{eq:ineqnoise4}
\frac{\theta}{2} \left(\delta_{k+1}-\frac{z^*}{(k+1)^{\alpha}}\right)^2 + \left(\delta_{k+1}-\frac{z^*}{(k+1)^{\alpha}}\right) \leq
\left(\delta_{k}-\frac{z^*}{k^{\alpha}}\right)
\end{eqnarray}holds, then it follows from Lemma \ref{prop:1} that
$$\delta_{k+1}-\frac{z^*}{(k+1)^{\alpha}} \leq \frac{\max\left\{\delta_1,\frac{4}{\theta}\right\}}{k+1},$$
and
{\begin{eqnarray}\label{eq:final}
\delta_{k+1}\leq \frac{z^* + \max\left\{\delta_1,\frac{4}{\theta}\right\}}{(k+1)^{\alpha}} < \frac{C}{(k+1)^{\alpha}}.
\end{eqnarray}}Clearly, \eqref{eq:final} contradicts \eqref{eq:nothold}, which implies that \eqref{deltaeta-global} is true. Next, we prove \eqref{eq:ineqnoise4}, which is equivalent to $$P(z^*):=\frac{\theta}{2} \left(\delta_{k+1}-\frac{z^*}{(k+1)^{\alpha}}\right)^2 + \left(\delta_{k+1}-\frac{z^*}{(k+1)^{\alpha}}\right)- \left(\delta_{k}-\frac{z^*}{k^{\alpha}}\right)\leq0.$$

We consider the following quadratic function $Q(z)$ with respect to $z:$
\begin{eqnarray*}\label{eq:quad}
Q(z) = \frac{\theta}{2}z^2 - \left(\frac{\theta C}{4}-1\right)z + \frac{3}{2}\left(\frac{3\theta \sigma}{2}+1\right)\sigma.
\end{eqnarray*}
It can be verified that the minimizer of $Q(z)$ is $z^*.$ Since the discriminant of $Q(z)$ is nonnegative, it follows that the minimum value
\begin{equation}\label{Qz*}
    Q(z^*)\leq 0.
\end{equation}
Moreover, for any $k\geq 1,$ we have $$\eqref{delta}\Longrightarrow \delta_{k+1}\leq \delta_k + \frac{3}{2}\eta_k\Longrightarrow\frac{1}{2}\delta_{k+1}^2 \leq \delta_k^2 + \frac{9}{4}\eta_k^2\Longrightarrow-\delta_k^2 \leq -\frac{1}{2}\delta_{k+1}^2 + \frac{9}{4}\sigma\eta_k.$$ Combining the last inequality in the above with \eqref{delta} yields
\begin{eqnarray}\label{eq:ineqnoise2}
{\frac{\theta}{2}\delta_{k+1}^2 + \delta_{k+1} \leq \delta_k  + \frac{3}{2}\left(\frac{3\theta\sigma}{2} + 1\right) \eta_k},
\end{eqnarray}
which further implies 
\begin{eqnarray}P(z^*)&=&\frac{\theta}{2} \delta_{k+1}^2 + \delta_{k+1} - \delta_k - \frac{\theta\delta_{k+1}z^*}{(k+1)^{\alpha}} + \frac{\theta(z^*)^2}{2(k+1)^{2\alpha}}
-\frac{z^*}{(k+1)^{\alpha}} + \frac{z^*}{k^{\alpha}}\nonumber\\
& \leq & \frac{3}{2}\left(\frac{3\theta\sigma}{2} + 1\right) \eta_k- \frac{\theta\delta_{k+1}z^*}{(k+1)^{\alpha}} + \frac{\theta(z^*)^2}{2(k+1)^{2\alpha}}
-\frac{z^*}{(k+1)^{\alpha}} + \frac{z^*}{k^{\alpha}}.\label{Q'z}
\end{eqnarray}
By ${\eta_k\leq \frac{\sigma}{k^{2\alpha}}},$ the assumption $\delta_{k+1} > \frac{C}{(k+1)^{\alpha}}$ (cf. \eqref{eq:nothold}), the facts $(k+1)^{2\alpha}\leq 4k^{2\alpha}$ and {$(k+1)^{\alpha}-k^{\alpha}\leq 1$ for all $k\geq 1$ and $\alpha\in(0,1]$}, we get
\begin{eqnarray*}
&& \frac{3}{2}\left(\frac{3\theta\sigma}{2} + 1\right) \eta_k- \frac{\theta\delta_{k+1}z^*}{(k+1)^{\alpha}} + \frac{\theta(z^*)^2}{2(k+1)^{2\alpha}}
-\frac{z^*}{(k+1)^{\alpha}} + \frac{z^*}{k^{\alpha}}\nonumber\\
&\leq & \frac{\frac{3}{2}\left(\frac{3\theta\sigma}{2} + 1\right) \sigma}{k^{2\alpha}} - \frac{\theta C}{4k^{2\alpha}}z^*
+ \frac{\theta}{2k^{2\alpha}}(z^*)^2 + \frac{1}{k^{2\alpha}}z^*\nonumber\\
&=&\frac{Q(z^*)}{k^{2\alpha}},
\end{eqnarray*}
which, together with \eqref{Qz*} and \eqref{Q'z}, yields $P(z^*)\leq 0.$

We now show \eqref{axketa}. From \eqref{axkineq}, we obtain
   \begin{align*}
   \left\|Ax^{k+1}-b\right\|^2 \leq  \frac{2}{{\beta}}\left(d(\lambda^{k+1})- d(\lambda^k)+\eta_k\right) \leq & \ \frac{2}{{\beta}}\left(d(\lambda^*)-d(\lambda^k)+\eta_k\right) \\
   = & \ \frac{2}{{\beta}}\left(\delta_k+\eta_k\right),
   \end{align*}
  which, together with \eqref{etakalpha} and \eqref{deltaeta-global}, yields \eqref{axketa}.

Next, we show \eqref{primalopt}. From the strong duality and the definition of $\LCal_{\beta}(x;\lambda)$ (cf. \eqref{augLag}), we obtain
  \begin{align*}
    F(x^*)\leq \LCal_{\beta}(x^{k+1};\lambda^*) & = F(x^{k+1}) + \left\langle \lambda^*, Ax^{k+1}-b \right\rangle + \frac{\beta}{2} \left\|Ax^{k+1}-b\right\|^2\\
    & \leq F(x^{k+1}) + \left\|\lambda^*\right\|\left\|Ax^{k+1}-b\right\| + \frac{\beta}{2} \left\|Ax^{k+1}-b\right\|^2.
  \end{align*}This, together with \eqref{axketa}, implies
  \begin{equation}\label{primalopt1}
    F(x^{k+1}) - F(x^*) \geq - \|\lambda^*\|\sqrt{\psi_k}  - \frac{\beta}{2}\psi_k.
  \end{equation} On the other hand, we have
  \begin{align}
    \LCal_{\beta}(x^{k+1};\lambda^k) &\leq \hat f_{\beta}(x(\lambda^k);\lambda_k)+\left\langle\nabla \hat f_{\beta}(x^{k+1};\lambda^k), x^{k+1}-x(\lambda^k)\right\rangle + g(x^{k+1})\nonumber\\
                                 & = d(\lambda^k)+\left\langle\nabla \hat f_{\beta}(x^{k+1};\lambda^k), x^{k+1}-x(\lambda^k)\right\rangle+g(x^{k+1})-g(x(\lambda^k))\nonumber\\
                                 &\leq d(\lambda^k) + \eta_k\nonumber \\
                                 &\leq F(x^*)+\eta_k ,\nonumber
  \end{align}
  where the first inequality is due to the convexity of $\hat f_{\beta}(x;\lambda)$ with respect to $x,$ the first equality is due to the definition of $d(\lambda^k),$ the second inequality is due to \eqref{violationcgal}, and the last inequality is due to
  the fact $d(\lambda^k)\leq F(x^*).$
  Recall the definition of $\LCal_{\beta}(x;\lambda),$ we get
  $$F(x^{k+1}) + \left\langle \lambda^k, Ax^{k+1}-b \right\rangle + \frac{\beta}{2} \left\|Ax^{k+1}-b\right\|^2\leq F(x^*)+\eta_k,$$ which, together with   \eqref{boundlambda}, immediately implies
  \begin{equation}\label{primalopt2}
    F(x^{k+1})- F(x^*) \leq \left(\|\lambda^*\|+B\right)\sqrt{\psi_k}+\eta_k.
  \end{equation}
  Combining \eqref{primalopt1} and \eqref{primalopt2} yields \eqref{primalopt}.

  {Finally, we show \eqref{Lagrange}. It follows from \eqref{lambda} and the definition of $\hat f_{\beta}(x; \lambda)$ in \eqref{hatf} that
  $${\nabla f(x^{k+1})}+A^T\lambda^{k+1}={\nabla f(x^{k+1})}+A^T\left(\lambda^{k}+\beta\left(Ax^{k+1}-b\right)\right)={\nabla \hat f_{\beta}(x^{k+1};\lambda^k)}.$$
  The above, together with \eqref{violationcgal}, immediately implies \eqref{Lagrange}.} The proof of Theorem \ref{thm-rate-global} is completed.\hfill$\Box$
\end{proof}


{
As a direct consequence of Theorem \ref{thm-rate-global}, we obtain the following result.
\begin{corollary}\label{corollary:rate}
  Let $\left\{x^k\right\}$ and $\left\{\lambda^k\right\}$ be generated by Algorithm \ref{alg:cgal} with $\eta_k=\frac{\sigma}{k^2}$. Then,
  \begin{equation}\label{rateO}\delta_k={\cal O}\left(\frac{1}{k}\right),~\left\|Ax^{k+1}-b\right\|={\cal O}\left(\frac{1}{\sqrt{k}}\right),~\left|F(x^{k+1})-F(x^*)\right|={\cal O}\left(\frac{1}{\sqrt{k}}\right)\end{equation}
  {and
  \begin{equation*}\max_{x\in\mathbb{R}^{n}} \left\{\left\langle{{\nabla f(x^{k+1})}+A^T\lambda^{k+1}},\,x^{k+1}-x\right\rangle + g(x^{k+1})-g(x)\right\}={\cal O}\left(\frac{1}{k^2}\right).\end{equation*}}
\end{corollary}
}

{Next, we present some iteration complexity results of Algorithm 1 to return an $\epsilon$-optimal solution of problem \eqref{problemcouple}. Our definition of the $\epsilon$-optimal solution is given as follows, which is a perturbation of the KKT optimality conditions. 
\begin{definition}[$\epsilon$-optimal solution]\label{esolution}
  For any given $\epsilon>0,$ $(x_{\epsilon}, \lambda_{\epsilon})$ is called an $\epsilon$-optimal solution pair if they satisfy
  \begin{equation}\label{oureoptimal1}
        \left\|Ax_{\epsilon}-b\right\|\leq \sqrt{\epsilon}
      \end{equation} and
        \begin{equation}\label{oureoptimal2}
        \max_{x\in\mathbb{R}^{n}} \left\{\left\langle{{\nabla f(x_{\epsilon})}+A^T\lambda_{\epsilon}},\,x_{\epsilon}-x\right\rangle + g(x_{\epsilon})-g(x)\right\}\leq \epsilon.
      \end{equation}
\end{definition}

\begin{theorem}[Iteration Complexity]\label{thmcomplexity}
For any $\epsilon>0,$ set \begin{equation}\label{alphasigma}\alpha=\sigma=1\end{equation} in \eqref{etakalpha} and the penalty parameter
\begin{equation}\label{betaopt}\beta=\frac{2(C+1)}{\sqrt{\epsilon}},\end{equation}
where $C$ is defined in \eqref{C}. Then, the total number of iterations for Algorithm 1, where AL subproblem \eqref{subproblemcgal} is approximately solved by {Algorithms 2 or 3} until a point $x^{k+1}$ satisfying \eqref{violationcgal} is found, to return an $\epsilon$-optimal solution of problem \eqref{problemcouple} satisfying \eqref{oureoptimal1} and \eqref{oureoptimal2} is at most
\begin{equation}\label{totaliteration1}
  T_1:={\left\lceil4\left(L_{f}+\frac{2\left(C+1\right)\|A\|^2}{\sqrt{\epsilon}}\right)\frac{D^2}{\epsilon^{3/2}}\right\rceil}
\end{equation} and
\begin{equation}\label{totaliteration2}
  T_2:=\left\lceil6\left(L_{f}+\frac{2\left(C+1\right)\|A\|^2}{\sqrt{\epsilon}}\right)\frac{D^2}{\epsilon^{3/2}}\right\rceil,
\end{equation} respectively, where $L_f$ is the Lipschitz constant of $\nabla f(x)$ and $D$ is defined in \eqref{Dx}.
\end{theorem}
\begin{proof} Let $K=\left\lceil 1/\sqrt{\epsilon}\right\rceil.$ Substituting $k=K,$ $\alpha~\text{and}~\sigma$ in \eqref{alphasigma}, and $\beta$ in \eqref{betaopt} into \eqref{axketa} and \eqref{Lagrange}, we immediately see that the pair $(x^{k+1}, \lambda^{k+1})$ satisfies \eqref{oureoptimal1} and \eqref{oureoptimal2}. Next, {we compute the total iteration complexity of Algorithm 1 with Algorithm 2 being used to solve the AL subproblem.} By invoking Theorem \ref{ThmNesInner}, we know that the total number of iterations is
$$\sum_{k=1}^{K}\left\{\left\lceil\frac{{4}L_{\hat f}D^2}{\eta_k}\right\rceil-1\right\}\leq \frac{{4}\left(L_{f}+\beta \|A\|^2\right)D^2}{\sigma}\sum_{k=1}^{K}k^{2\alpha}\leq T_1.$$ Using the same argument, we can show that the iteration complexity of {Algorithm 1} {with Algorithm 3} being used to solve the AL subproblem is upper bounded by $T_2.$ {We omit the details for succinctness.}  
\hfill$\Box$
\end{proof}

}

{\section{Comparisons with existing works \cite{lan2015iteration,necoara152}}
In this section, we make some remarks on the comparison of our proposed IAL framework (Algorithm \ref{alg:cgal}) and two closely related methods in \cite{lan2015iteration,necoara152} for solving the linearly constrained convex programming problems.

We first compare our proposed IAL method with the one in \cite{lan2015iteration}. The method in \cite{lan2015iteration} is designed for solving problem \eqref{problemcouple} with $g(x)=\text{Ind}_{\X}(x).$ It applies Nesterov's optimal first-order method to solve AL subproblem \eqref{subproblemcgal} until a point $x^{k+1}$ satisfying \eqref{neseta} is found. Our IAL framework can be used to solve more general problem \eqref{problemcouple} (with a general composite function $g(x)$). Our framework requires approximately solving subproblem \eqref{subproblemcgal} until a point $x^{k+1}$ satisfying \eqref{violationcgal} (which is easier to check than \eqref{neseta}) is found.

The work \cite{lan2015iteration} shows the same non-ergodic convergence rate results as ours {in \eqref{rateO}} in Corollary \ref{corollary:rate}, but under a much stronger condition that the sequence $\left\{\eta_k\right\}$ in \eqref{neseta} satisfies \begin{equation*}\label{lanassumption}\sum_{i=1}^k\eta_i^2={\cal O}\left(\frac{1}{k}\right).\end{equation*} To make it more clearly, consider the special case where we are interested in finding an exact solution of problem \eqref{problemcouple}, which requires $k\rightarrow +\infty$ in Corollary \ref{corollary:rate}. In this case, the method in \cite{lan2015iteration} needs to solve each AL subproblem exactly (i.e., $\eta_i$ in \eqref{neseta} needs to be zero for all $i=1,2,\ldots$), while our IAL framework only needs to solve each subproblem approximately (i.e., $\eta_i$ in \eqref{violationcgal} only needs to be in the order of ${\cal O}(1/i^2)$ for $i=1,2,\ldots$).

{Next, we compare our proposed IAL method with the one in \cite{necoara152}. The closest related method in \cite{necoara152} to our IAL method, called inexact gradient augmented Lagrangian, is designed for solving a class of convex conic problems{
\begin{equation}\label{conicproblem}
\begin{array}{cl}
\displaystyle \min_{u\in\cal{U}}~\displaystyle f(u),~\mbox{s.t.}\quad Gu+g\in\cal{K},
\end{array}
\end{equation} where ${\mathcal{U}}\subseteq \mathbb{R}^n$ is a convex compact set, ${\mathcal{K}}\subseteq \mathbb{R}^m$ is a convex cone, $G\in \mathbb{R}^{m\times n},$ and $g\in\mathbb{R}^{m}.$ It is easy to show that problem \eqref{conicproblem} is a special case of problem \eqref{problemcouple}.} At the $k$-th iteration, the method in \cite{necoara152} applies Nesterov's optimal first-order method to solve AL subproblem \eqref{subproblemcgal} until a point $x^{k+1}$ satisfying
\begin{equation}\label{Necx}\LCal_\beta\left(x^{k+1};\lambda^k\right)-\LCal_\beta\left(x(\lambda^k);\lambda^k\right)\leq \delta\end{equation}
 is found, where $\delta>0$ is the given accuracy; then the method updates the dual variable by
 \begin{equation}\label{Neclambda}\lambda^{k+1}=\lambda^k+\frac{\beta}{2}\left(Ax^{k+1}-b\right).\end{equation} Again, our framework requires approximately solving subproblem \eqref{subproblemcgal} until a point $x^{k+1}$ satisfying \eqref{violationcgal} is found and updates the dual variable via \eqref{lambda}. Notice that the dual variable update formula \eqref{Neclambda} in \cite{necoara152} is different from \eqref{lambda} in the classical AL method. 

The work \cite[Corollary 3.3 and Theorem 3.4]{necoara152} shows the ergodic convergence rate results in terms of the dual function values, the primal infeasibility, and the primal function values. Moreover, the work \cite[Theorem 3.5]{necoara152} shows that it takes the algorithm (with an optimal choice of the penalty parameter $\beta$ and the solution tolerance $\delta$) a total number of ${\cal O}\left(1/\epsilon\right)$  iterations to return an $\epsilon$-optimal solution $u_{\epsilon}$ defined as follows:
\begin{equation}\label{esolution}|f(u_{\epsilon})-f^*|\leq \epsilon~\text{and}~\text{dist}_{\cal K}(G u_{\epsilon}+g)\leq \epsilon,\end{equation} {where $f^*$ is the optimal value of problem \eqref{conicproblem}.} As mentioned in \cite{necoara152}, the algorithm behind the above complexity result reduces to a quadratic penalty method (without any update of the dual variable) and therefore there is {no convergence guarantee for the dual variable.}

In summary, the results in our paper significantly differ from the ones in \cite{necoara152} in terms of global convergence results, convergence rate results, and the algorithms. \begin{itemize}
  \item [-] \textbf{Global convergence}. {Our IAL framework enjoys the global convergence (under the assumption that the error sequence $\left\{\eta_k\right\}$ is summable), but the inexact gradient augmented Lagrangian method in \cite{necoara152} does not have global convergence guarantee, due to the existence of the positive accuracy constant $\delta.$}
  \item [-] \textbf{Convergence rate}. The convergence results for our IAL framework in terms of the dual objective values, the primal infeasibility, and the primal objective values are all for the \emph{non-ergodic} solution, but the results for the inexact gradient augmented Lagrangian method in \cite{necoara152} are for the \emph{ergodic} solution.
  \item [-] \textbf{Algorithms and dual optimality guarantee}. The dual variable update formula in our IAL framework and the one in \cite{necoara152} are different. The algorithms behind the iteration complexity results (see Theorem 3.8 in \cite{necoara152} and Theorem \ref{thmcomplexity} in our paper) are also sharply different from each other. The algorithm behind Theorem 3.8 in \cite{necoara152} is essentially a penalty method (without any update of the dual variable) but the algorithm behind Theorem \ref{thmcomplexity} in our paper indeed is an IAL method. Therefore, there is no convergence guarantee for the dual variable in \cite{necoara152}, {while it is guaranteed in our paper. In particular,} from {Theorem \ref{thm-rate-global}} and with the choice of the parameters in Theorem \ref{thmcomplexity}, we get
$$\max_{x\in\mathbb{R}^{n}} \left\{\left\langle{{\nabla f(x^{k+1})}+A^T\lambda^{k+1}},\,x^{k+1}-x\right\rangle + g(x^{k+1})-g(x)\right\}\leq \epsilon$$
and $${d(\lambda^*)-d(\lambda^{k+1})}={\cal O}\left(\sqrt{\epsilon}\right).$$
\item [-] \textbf{Definition of $\epsilon$-optimal solution.} Our definition of the $\epsilon$-optimal solution is a natural perturbation of the KKT optimality conditions of problem \eqref{problemcouple}, which involves the dual variable. The definition of the $\epsilon$-optimal solution in \cite{necoara152}, {i.e., \eqref{esolution},} does not involve the dual variable.
\end{itemize}
}

%

{\section{Numerical results}\label{sec:simulation}

In this section, we present some preliminary numerical results for the purpose of comparing the following two things: (i) the difference of the ergodic and non-ergidoc solutions; (ii) the difference of the ``exact'' augmented Lagrangian (EAL) method and our IAL method. Our codes were written in MATLAB and the results were obtained on a standard PC.

The numerical experiments were conducted on basis pursuit problem \eqref{problemcs}. In Table \ref{table:IAL-vs-EAL} we report the results for some small problem with $m=60$, $n=100$ {and sparsity} (number of nonzero entries) $s=15$.
The 10 instances were randomly generated in the following manner: the entries of $A$ were generated randomly following the standard Gaussian distribution $\mathcal{N}(0,1)$; the positions of the nonzero entries in $x^*$ were uniformly randomly chosen and their values were generated following the uniform distribution in $(0,1)$; and finally $b$ is set to $b=Ax^*.$ To construct the bounded set containing the true solution, we set $\hat{x} = A_m^{-1}b$, where $A_m$ is the square matrix formed by the first $m$ columns of $A$. We use $\bar{x}$ to denote the solution returned by EAL or IAL.

We ran both of IAL and EAL for $K=200$ (dual) iterations. We set both of the initial (primal) point $x^1$ and the initial (dual) Lagrange multiplier $\lambda^1$ to $0$. We applied the proximal gradient method to solve the {AL} subproblem until \eqref{violationcgal} is satisfied with $\eta_k =1/k^2$ for IAL and $\eta_k\equiv 10^{-4}$ for EAL. We reported the comparison results in Table \ref{table:IAL-vs-EAL}. In particular, we reported the cpu time (in seconds), the relative error of the solution (denoted by relerr $= \|\bar{x}-x^*\|/\|x^*\|$), the residual of the linear constraint (denoted by resi $=\|A\bar{x}-b\|$), and the objective value error (denoted by objerr $= |\|\bar{x}\|_1-\|x^*\|_1|$). Moreover, we also reported the sparsity (the number of the nonzero entries) of the returned non-ergodic solution $\bar{x}$ (denoted by $s_n$) and the sparsity of the ergodic solution $x_e := \sum_{k=1}^K x^k/K$ (denoted by $s_e$). 

We see from Table \ref{table:IAL-vs-EAL} that IAL and EAL are comparable in terms of the solution quality measured by relerr, resi, and objerr, and there is no evidence showing that one is better than the other. However, IAL is much faster than EAL in terms of the cpu time. This is expected because the {AL} subproblems are solved much less accurately in IAL in the first 100 iterations (compared to EAL). It is worth mentioning that $\eta_k <10^{-4}$ for $k>100$ and $\eta_{k}=0.25*10^{-4}$ for $k=200$ in IAL, that is, the last 100 {AL} subproblems in IAL are solved slightly more accurately than EAL, but IAL is still much faster.

We also observe from Table \ref{table:IAL-vs-EAL} that for both IAL and EAL, the non-ergodic solution $\bar x$ is significantly more sparse than the ergodic solution $x_e$. In fact, the sparsity pattern of the non-ergodic solution always perfectly matches that of the true solution $x^*$. This well justifies the importance of our global convergence and convergence rate analysis on the \emph{non-ergodic} solution in this paper.

{We also report numerial results for some larger problem instances in Table \ref{table:IAL-vs-EAL-large}. The problem instances are randomly generated in the same manner as before. From Table \ref{table:IAL-vs-EAL-large} we have the same observations as the ones from Table \ref{table:IAL-vs-EAL}: IAL is much faster than EAL in terms of the cpu time and the non-ergodic solutions are significantly more sparse than the ergodic solutions.}

\begin{table}[h]
	\centering
	\caption{Numerical results of IAL and EAL for solving basis pursuit problem \eqref{problemcs} {with $m=60$, $n=100$, $s=15$}, where ID denotes the problem instance index.}
	\label{table:IAL-vs-EAL}
	\begin{tabular}{|*{13}{c|}}
		\hline
		    & \multicolumn{6}{|c|}{IAL} & \multicolumn{6}{|c|}{EAL} \\ \hline
            ID & $s_e$ & $s_n$ & relerr & resi & objerr & cpu & $s_e$ & $s_n$ & relerr & resi & objerr & cpu \\\hline
		1 & 40 & 15 & 2.0e-08 & 1.9e-07 & 2.8e-08 & 2.2 & 40 & 15 & 1.4e-08 & 1.3e-07 & 5.7e-08 & 3.9 \\\hline
2 & 27 & 15 & 1.0e-08 & 9.6e-08 & 1.2e-08 & 0.6 & 27 & 15 & 8.0e-09 & 7.0e-08 & 1.6e-08 & 1.3 \\\hline
3 & 16 & 15 & 3.3e-08 & 3.5e-07 & 1.2e-08 & 0.2 & 16 & 15 & 3.9e-08 & 4.0e-07 & 1.4e-08 & 0.3 \\\hline
4 & 23 & 15 & 7.2e-09 & 8.1e-08 & 2.7e-08 & 0.3 & 23 & 15 & 9.0e-09 & 9.8e-08 & 3.0e-08 & 0.5 \\\hline
5 & 25 & 15 & 1.7e-08 & 1.7e-07 & 9.1e-09 & 0.4 & 25 & 15 & 6.4e-09 & 5.7e-08 & 3.1e-09 & 0.8 \\\hline
6 & 26 & 15 & 3.0e-09 & 2.2e-08 & 8.5e-10 & 1.1 & 26 & 15 & 2.8e-09 & 2.1e-08 & 1.8e-09 & 2.0 \\\hline
7 & 26 & 15 & 5.3e-08 & 4.7e-07 & 1.7e-07 & 0.7 & 26 & 15 & 3.5e-08 & 3.0e-07 & 1.3e-07 & 1.7 \\\hline
8 & 23 & 15 & 2.9e-08 & 2.7e-07 & 2.5e-08 & 0.3 & 22 & 15 & 1.1e-08 & 1.0e-07 & 4.6e-08 & 0.7 \\\hline
9 & 25 & 18 & 6.0e-04 & 6.7e-03 & 1.3e-03 & 0.5 & 25 & 18 & 6.0e-04 & 6.7e-03 & 1.3e-03 & 1.0 \\\hline
10 & 17 & 15 & 6.4e-08 & 6.8e-07 & 7.6e-08 & 0.2 & 17 & 15 & 4.7e-08 & 4.6e-07 & 3.0e-08 & 0.3 \\\hline
	\end{tabular}

\end{table}
}

{\color{blue}
\begin{table}[h]
	\centering
	\caption{{Numerical results of IAL and EAL for solving larger basis pursuit problem \eqref{problemcs}, where ID denotes the problem instance index.}}
	\label{table:IAL-vs-EAL-large}
	\begin{tabular}{|*{13}{c|}}
		\hline
		    & \multicolumn{6}{|c|}{IAL} & \multicolumn{6}{|c|}{EAL} \\ \hline
& \multicolumn{12}{|c|}{{$m=600$, $n=1000$, $s=150$}}\\\hline
            ID & $s_e$ & $s_n$ & relerr & resi & objerr & cpu & $s_e$ & $s_n$ & relerr & resi & objerr & cpu \\\hline
		1 & 241 & 150 & 5.3e-12 & 5.2e-10 & 4.1e-11 & 588.5 & 241 & 150 & 3.9e-12 & 3.9e-10 & 2.0e-11 & 871.5 \\\hline
2 & 188 & 150 & 7.4e-11 & 7.1e-09 & 3.7e-10 & 139.6 & 188 & 150 & 6.8e-11 & 6.7e-09 & 3.4e-10 & 241.8 \\\hline
3 & 233 & 150 & 5.4e-11 & 5.1e-09 & 5.2e-10 & 239.0 & 233 & 150 & 2.5e-11 & 2.4e-09 & 1.4e-10 & 360.6 \\\hline
4 & 222 & 150 & 6.0e-11 & 5.5e-09 & 3.2e-10 & 130.5 & 222 & 150 & 3.1e-11 & 3.0e-09 & 3.1e-10 & 222.0 \\\hline
5 & 225 & 150 & 1.2e-11 & 1.2e-09 & 2.8e-11 & 452.9 & 225 & 150 & 6.4e-12 & 6.0e-10 & 5.3e-11 & 698.7 \\\hline
& \multicolumn{12}{|c|}{{$m=1800$, $n=3000$, $s=450$}}\\\hline
1 & 1609 & 450 & 7.4e-12 & 2.2e-09 & 1.2e-10 & 2999.1 & 1609 & 450 & 6.8e-12 & 2.0e-09 & 7.8e-12 & 4424.2 \\\hline
2 & 1668 & 450 & 3.6e-12 & 9.9e-10 & 3.0e-11 & 3823.4 & 1668 & 451 & 4.2e-12 & 1.2e-09 & 2.2e-11 & 8075.7 \\\hline
3 & 1639 & 450 & 1.3e-11 & 3.5e-09 & 5.3e-11 & 5176.9 & 1639 & 450 & 8.6e-12 & 2.5e-09 & 5.7e-11 & 5921.8 \\\hline
4 & 1669 & 450 & 5.8e-12 & 1.6e-09 & 2.1e-11 & 7136.7 & 1669 & 451 & 7.9e-12 & 2.4e-09 & 1.1e-10 & 7858.3 \\\hline
5 & 1692 & 450 & 2.9e-13 & 7.9e-11 & 5.1e-13 & 2464.4 & 1692 & 450 & 3.1e-13 & 8.7e-11 & 6.8e-12 & 3146.6 \\\hline
	\end{tabular}

\end{table}
}

}

\begin{acknowledgements}
{We are grateful to the editors and three anonymous referees for their insightful comments and suggestions that helped us improve the quality of the paper greatly.}
We would like also to thank {Guanghui Lan,} Zhaosong Lu, Zaiwen Wen, and Wotao Yin for their insightful comments, which helped us in improving the results in this paper. {We thank Xiangfeng Wang for the useful discussion on an earlier version of this paper.}

{The work of Y.-F. Liu was supported in part by NSFC grants 11631013, 11331012, 11671419, and 11571221, and Beijing NSF grant L172020. The work of X. Liu was supported in part by NSFC grants 11622112, 11471325, 91530204, and 11688101, the National Center for Mathematics and Interdisciplinary Sciences, CAS, and Key Research Program of Frontier Sciences (QYZDJ-SSW-SYS010), CAS. The work of S. Ma was supported in part by a startup package in Department of Mathematics at UC Davis.}
\end{acknowledgements}



\bibliography{ial20180312}
\bibliographystyle{plain}

\end{document}